\documentclass[10pt,a4paper]{article}
\usepackage{amsmath}
\usepackage{amssymb}
\usepackage{amsthm}
\usepackage{enumerate}
\usepackage{color}
\usepackage[english]{babel}
\usepackage{hyperref}

\newtheorem{theorem}{Theorem}[section]
\newtheorem{corollary}[theorem]{Corollary}
\newtheorem{lemma}[theorem]{Lemma}

\newtheorem{remark}[theorem]{Remark}
\newtheorem{example}[theorem]{Example}
\newtheorem{definition}[theorem]{Definition}

\newcommand{\st}{\stackrel{*}\le} % star order
\newcommand{\lstar}{{}_{\ast}{\le}} % left star order
\newcommand{\rstar}{{\le}_{*}} % right star order
\newcommand{\minus}{\stackrel{-}\le} % minus order
\newcommand{\spa}{\stackrel{s}\preceq} %  space order
\newcommand{\diam}{\stackrel{\diamond}\le} % diamond order
\newcommand{\C}{{\ensuremath{\mathbb{C}}}}
\newcommand{\Cm}{{\ensuremath{\C^{m\times n}}}}
\newcommand{\Cn}{{\ensuremath{\C^{n\times m}}}}
\newcommand{\Cnn}{{\ensuremath{\C^{n\times n}}}}

\newcommand{\Crr}{{\ensuremath{\C^{r\times r}}}}
\newcommand{\Ra}{{\ensuremath{\cal R}}}
\newcommand{\Nu}{{\ensuremath{\cal N}}}

\newcommand{\rk}{{\ensuremath{\rm rk}}}
\newcommand{\ra}{{\ensuremath{\text{\rm rk}}}}
\usepackage[top=2.5cm, left=2.0cm, right=2.0cm, bottom=2.5cm]{geometry}

\begin{document}
\author{D.E. Ferreyra\thanks{Universidad Nacional de R\'io Cuarto, CONICET, FCEFQyN, RN 36 Km 601, R\'io Cuarto, 5800, C\'ordoba. Argentina. E-mail:\texttt{deferreyra@exa.unrc.edu.ar, flevis@exa.unrc.edu.ar, vorquera@exa.unrc.edu.ar}.}\;, F.E. Levis$^{*}$, G. Maharana\thanks{ Department of
Mathematics, Siksha ‘O’ Anusandhan
(Deemed to be University), Khandagiri
Square, Bhubaneswar 751030, Odisha,
India. E-mail: \texttt{gayatrimaharana@soa.ac.in}}\;,
V. Orquera$^{*}$
}


\title{New characterizations of the diamond partial order}
\date{}
\maketitle

\begin{abstract}
Baksalary and Hauke introduced the diamond partial order in 1990, which we revisit in this paper. This order was defined on the set of rectangular matrices and is the same as the star and minus partial orders for partial isometries. New ways of describing and studying the diamond partial order are being looked into in this paper. Particularly, we present a new characterization by using an additivity property of the column spaces. Additionally, we also study the relationship between the left (resp., right) star and diamond partial orders. Specifically, we obtain conditions in which the diamond partial order means the left (resp., right) star partial order. The reverse order law for the Moore-Penrose inverse is characterized when $A$ is below $B$ under the diamond partial order. Finally, an interesting way of describing bi-dagger matrices is found.
\end{abstract}



\maketitle

\section{Introduction}
The area of matrix generalized inverses has grown steadily over the last few decades, which plays a crucial role in the study of partial orders \cite{MiBhMa}. The minus partial order was introduced by Hartwing \cite{Ha} in 1980, which is a generalization of the star partial order defined by Drazin \cite{Dr2} two years earlier. Later, Mitra \cite{Mi} defined the sharp partial order as a binary relation on the set of matrices of index at most one, using the group inverse.

In 1990, Baksalary and Hauke \cite{BaHa} introduced a new partial order on the set of rectangular matrices. This order coincides with the star and minus partial orders on the set of partial isometries and is now known as the diamond partial order \cite{LePaTh}. The partial diamond order hasn't been studied as much as the star, minus, and sharp orders \cite{BaBaLi1, BaBaLi2, BaHaLiLi, BaMi, HaSt}.

This article represents a modest contribution in this direction as we explore new properties and characterizations of the diamond partial order. The following is the outline of the paper: We present the basic notations, definitions, and some known results of generalized inverses and partial orders in Section 2. In Section 3, we establish new characterizations of the diamond partial order. In particular, we characterize the diamond partial order by means of an additive property of the column spaces. In Section 4, we analyze the relationships between the left (resp., right) star and diamond partial orders. We conclude the work by investigating the reverse order law for the Moore-Penrose inverse when $A$ is below $B$ under the diamond partial order, as well as establishing an interesting characterization of bi-dagger matrices.

\section{Notation and Preliminaries}

For convenience, throughout this paper, $ \mathbb{C}^{m\times n}$ stands for the set of $m \times n$ matrices over complex numbers. In addition, for $A\in\Cm$, the symbols $A^*$,$A^{-1}$, $\rk(A)$, $\Nu(A)$, and $\Ra(A)$ will stand for the conjugate transpose, the inverse (when $m=n$), the rank, the null space, and the range space of $A$, respectively. Moreover, $I_n$ will refer to the $n \times n$ identity matrix. The index of $A\in \mathbb{C}^{n\times n}$ is the smallest non-negative integer $k$, such that $\ra(A^k)=\ra(A^{k+1})$, which is denoted by $ind(A)$. A {\it inner inverse} of $A\in \Cm$ is defined as a matrix $X \in \Cn$ satisfying the equation $AXA = A$, and the set of all inner inverses of $A$ is denoted by $A\{1\}$. 

The {\it Moore-Penrose inverse} of $A \in \Cm$ is the unique matrix $X\in \Cn$ that satisfies the Penrose equations $AXA=A,~XAX=X,~(AX)^*=AX$, and $(XA)^*=XA$. It is denoted by $A^{\dagger}$.

The matrices $P_A:=AA^{\dag}$ and $Q_A:=A^{\dag}A $ denote the orthogonal projectors onto $\Ra(A)$ and $\Ra(A^*)$, respectively.

Recall that a binary relation on a nonempty set that is reflexive and transitive is called a {\it pre-order}, while a {\it partial order} is a pre-order that also satisfies the antisymmetric property. 

The following binary relations are well known. For matrices $A,B\in \Cm$, we say: 
\begin{enumerate}
\item[$\bullet$]$A$ is below $B$ under the {\it star partial order}, i.e., $A\st B$, if $A^* A = A^* B$ and $AA^*=B A^*$ \cite{Dr2}.
\item[$\bullet$] $A$ is below $B$ under the {\it minus partial order}, i.e., $A\minus B$, if $A^- A = A^- B$ and $AA^-=B A^-$, for some $A^-\in A\{1\}$ \cite{Ha}.
\item[$\bullet$] $A$ is below $B$ under the {\it space pre-order},  i.e., $A\spa B$, if $\Ra(A)\subseteq \Ra(B)$ and $\Ra(A^*)\subseteq \Ra(B^*)$ \cite{MiBhMa}.
\end{enumerate}
Out of several known characterizations of ordering mentioned above, the following will be referred to in the sequel:
\begin{eqnarray}
A \st B &~\Leftrightarrow~& A^\dag A= A^\dag B ~~\text{and}~~ AA^\dag=BA^\dag. \label{star order} \\
A \minus B &~\Leftrightarrow~& \ra(B-A)=\ra(B)-\ra(A) ~\Leftrightarrow~ B\{1\}\subseteq A\{1\}. \label{minus order} \\
A \spa B &~\Leftrightarrow~& A=BB^{-}A=AB^{-}B, \text{ for each } B^{-} \in B\{1\}. \label{space order}
\end{eqnarray}

An interesting property of the relations defined in \eqref{star order}, \eqref{minus order}, and \eqref{space order} is given by the following chains of implications:

\begin{equation}\label{star minus space}
 A \st B  ~\Rightarrow~ A \minus B ~\Rightarrow~ A \spa B.
\end{equation}

In \cite{BaHa}, Baksalary and Hauke have defined the following order relation (currently called diamond partial order \cite{LePaTh}):
 
 \begin{definition}\label{def BaHa} Let  $A,B\in \Cm.$ Then, we say 
 $A\diam B$ if $A\spa B$ and $AB^*A=AA^*A$.
 \end{definition}
 
As $AB^*A=AA^*A$ is equivalent to $A^*(B-A)A^*=0$, from the definition of star partial order, it's clear 
\begin{equation}\label{star diam space}
A\st B \Rightarrow A\diam B \Rightarrow A\spa B.
\end{equation}
It is well known that the chains \eqref{star minus space} and \eqref{star diam space} cannot be dovetailed. Moreover, $A\diam B \not\Rightarrow A\minus B$ and $A\minus B \not\Rightarrow A\diam B$, as shown in the examples below,
\begin{example}\label{example 1}
Let 
\begin{equation*} 
A=\left[\begin{array}{cc}
0 & 1\\
0 &0
\end{array}\right] \quad \text{and} \quad B=\left[\begin{array}{cc}
1 & 1\\
0 &1
\end{array}\right].
\end{equation*} 
As $B$ is nonsingular, from \eqref{space order} it is clear that $\Ra(A)\subseteq \Ra(B)$ and $\Ra(A^*)\subseteq \Ra(B^*)$, and therefore $A\spa B$. Also, as $B-A=I_2$, we have $A^*(B-A)A^*=(A^*)^2=0$, whence $AB^*A=AA^*A$. So, $A\diam B$. However, $\ra(B-A)=\ra(I_2)=2>\ra(B)-\ra(A)$. Thus, from \eqref{minus order} it follows that $A\not \minus B$.
\end{example}
 \begin{example}\label{example 2}
Let \begin{equation*}
A=\left[\begin{array}{cc}
0 & 1\\
0 &0
\end{array}\right] \quad \text{and} \quad B=\left[\begin{array}{cc}
1 & 0\\
-1 &1
\end{array}\right].
\end{equation*}  
Clearly, $\ra(B-A)=\ra(B)-\ra(A)$ and so $A\minus B$. However, direct calculations show that $AB^*A\neq AA^*A$, whence $A\not \diam B$.	
    \end{example}

As a counterpart to the above examples, in \cite{BaHa} it was proved that there are interesting relationships between these orders: 
\begin{equation}\label{diamond characterization}
A\diam B \Leftrightarrow A^\dag \minus B^\dag. 
\end{equation} 
From \eqref{star order} and \eqref{diamond characterization}, it follows that the orders $\st$, $\diam$, and $\minus$ coincide on the set of partial isometries (that is, $A^\dag=A^*$).

\section{Some new characterizations of the diamond partial order}

In this section, we establish new characterizations of the diamond partial order. In particular, we characterize the diamond partial order by means of an additive property of the column spaces.

Besides that, we begin with an auxiliary lemma.

\begin{lemma}\label{lemma rango} Let $A,B\in \Cm$. Then, $\Ra(A+B)=\Ra(A)+\Ra(B)$ if and only if $\Ra(A)\subseteq \Ra(A+B)$. 
\end{lemma}

\begin{theorem} \label{characterization 1} Let $A,B\in\Cm$. Then the following statements are equivalent:
\begin{enumerate}[(a)]
\item $A\diam B$;
\item $\Ra(B^*)=\Ra(A^*)\oplus \Ra(B^\dagger -A^\dagger )$;
\item $\Ra(A^*)\cap \Ra(B^\dagger -A^\dagger )=\{0\}$ and $\Ra(A^*)\subseteq \Ra(B^*)$.
\end{enumerate}
\end{theorem}
\begin{proof}
(a)$\Rightarrow$(b).  From \eqref{star diam space} we know that $A\spa B$, and so  $\Ra(A^*)\subseteq \Ra(B^*)$. Now, Lemma \ref{lemma rango} implies 
\begin{equation}\label{eq suma algebraica}
\Ra(B^*)=\Ra(B^\dagger)=\Ra(A^\dagger)+\Ra(B^\dagger-A^\dagger)=\Ra(A^*)+\Ra(B^\dagger-A^\dagger).
\end{equation}
Moreover, from  \eqref{diamond characterization} we have $A^\dagger \minus B^\dagger$. So,   $(A^\dagger)^-A^\dagger=(A^\dagger)^-B^\dagger$ for some inner inverse $(A^\dag)^-$ of $A^\dag$, which in turn is equivalent to $\Ra(B^\dag-A^\dag)\subseteq \Nu((A^\dag)^-)=\Nu(A^\dag(A^\dag)^-)$. In consequence, as $A^\dag(A^\dag)^-$ is idempotent,  $\Nu(A^\dag(A^\dag)^-)\cap \Ra(A^\dag(A^\dag)^-)=\{0\}$. Thus,
\begin{equation}\label{eq inter}
\begin{split}
\Ra(A^*)\cap \Ra(B^\dag-A^\dagger) & =\Ra(A^\dag)\cap \Ra(B^\dag-A^\dag)=\Ra(A^\dag(A^\dag)^-)\cap \Ra(B^\dag-A^\dag) \\ & \subseteq \Nu(A^\dag(A^\dag)^-)\cap \Ra(A^\dag(A^\dag)^-)=\{0\}.
\end{split}
\end{equation}
Now, the implication follows from \eqref{eq suma algebraica} and \eqref{eq inter}. \\
(b)$\Rightarrow$(c) Trivial. \\
(c)$\Rightarrow$(a)  From $\Ra(A^\dag)\subseteq \Ra(B^\dag)$ we have $A^\dag = B^\dag (B^\dag)^-  A^\dag $, where $(B^\dag)^-$ is an arbitrary inner inverse of $B^\dag$. Thus, $A^\dag=A^\dag (B^\dag)^-A^\dag+(B^\dag-A^\dag)(B^\dag)^-  A^\dag,$
whence 
$\Ra(A^\dag-A^\dag (B^\dag)^-A^\dag)\subseteq \Ra(B^\dag -A^\dag)$. In consequence, $\Ra(A^\dag-A^\dag (B^\dag)^-A^\dag)\subseteq \Ra(A^\dag)\cap \Ra(B^\dag-A^\dag)=\Ra(A^*)\cap \Ra(B^\dagger-A^\dagger)=\{0\}$, and so  $A^\dag=A^\dag (B^\dag)^-A^\dag$. Therefore, $B^\dag\{1\}\subseteq A^\dag\{1\}$ which is equivalent to 
 $A^\dag \minus B^\dag$ or equivalently $A\diam B$ by \eqref{diamond characterization}.
\end{proof}

\begin{corollary}  Let $A,B\in \Cm$. Then the following statements are equivalent:
\begin{enumerate}[(a)]
\item $A\diam B$;
\item $\ra(B^\dag -A^\dag)=\ra((I_n-Q_A)B^\dag)$ and $\Ra(A^*)\subseteq \Ra(B^*)$;
\item $A^\dag=QB^\dag $ for some idempotent matrix  $Q$ and $\Ra(A^*)\subseteq \Ra(B^*)$.
\end{enumerate}
\end{corollary}
\begin{proof}
(a)$\Leftrightarrow$(b)
Firstly, we note that 
$(I_n-Q_A)B^\dag=(I_n-Q_A)(B^\dag -A^\dag)$, whence 
\begin{equation}\label{rank formula}
\begin{split}
\ra((I_n-Q_A)B^\dag)&= \ra(B^\dag-A^\dag)-\dim\left(\Ra(B^\dag-A^\dag)\cap \Nu(I_n-Q_A)\right)\\
&=\ra(B^\dag-A^\dag)-\dim\left(\Ra(B^\dag-A^\dag)\cap \Ra(A^\dag)\right)\\
&= \ra(B^\dag-A^\dag)-\dim\left(\Ra(B^\dag-A^\dag)\cap \Ra(A^*)\right).
\end{split}
\end{equation}
Now, the equivalence  is clear from \eqref{rank formula} and  Theorem \ref{characterization 1}.\\
(a)$\Rightarrow$(c) From \eqref{diamond characterization} we have  $(A^\dagger)^-A^\dagger=(A^\dagger)^-B^\dagger$ for some  $(A^\dag)^-\in A^\dag\{1\}$. Thus, $A^\dag=A^\dag (A^\dag)^-B^\dag$. Now, it is sufficient to choose $Q=A^\dag (A^\dag)^-$ which clearly is idempotent. Obviously, $\Ra(A^*)\subseteq \Ra(B^*)$ because $A\spa B$. \\
(c)$\Rightarrow$(a) Suppose $A^\dag=QB^\dag$, where $Q$ is idempotent. Therefore, $\Ra(B^\dag-A^\dag)\subseteq \Ra(I_n-Q)=\Nu(Q)$. Moreover, $\Ra(A^\dag)\subseteq \Ra(Q)$. Thus, $\Ra(A^\dag)\cap \Ra(B^\dag -A^\dag)\subseteq \Ra(Q)\cap \Nu(Q)=\{0\}$. Now, Theorem \ref{characterization 1} completes the proof. 
\end{proof}

We now provide a new canonical form of the diamond partial order. 
The  tools we use is the classical Singular Value Decomposition (SVD) and  one of its most important consequences for the case of a square matrix called  Hartwig-Spindelböck decomposition \cite[Corollary 6]{HaSp}. 

\begin{theorem}\label{theorem canonical} Let $A,B$ $\in \Cm$. Then the following statements
are equivalent:
\begin{enumerate}[a)]
\item $A\diam B$;
\item There exist two unitary matrices $U$ and $V$  such that
\begin{equation*} 
A=U\begin{bmatrix}
C_1 & C_2 & 0\\
                       0 & 0 & 0 \\
                       0 & 0 & 0
                     \end{bmatrix}V^* \quad\text{and}\quad
    B=U\begin{bmatrix}
                       D_1 & D_2 & 0\\
                       D_3 & D_4 & 0 \\
                       0 & 0 & 0
                     \end{bmatrix}V^*,
\end{equation*}
where $C_1 C_1^*+C_2 C_2^*=C_1 D_1^*+C_2 D_2^*$ is nonsingular and $D_3$, $D_4$ are arbitrary matrices of adequate size satisfying $\rk(B)= \rk\left( \begin{bmatrix}
                       D_1 & D_2 \\
                       D_3 & D_4 
                     \end{bmatrix}\right)$. 
\end{enumerate}
\end{theorem}
\begin{proof}
(a)$\Rightarrow$(b) Since $A\diam B$ we have $A\spa B$ and $AA^*A=AB^*A$. 
 Let $B=U_1\begin{bmatrix}
       D & 0\\
       0 & 0\\
     \end{bmatrix}V_1^*$ be a SVD of $B$, where $D$ is a  positive definite diagonal matrix of order $r=\rk(B)$ and $U_1,V_1$ are unitary. Since
\begin{equation} \label{charact ord space}
\Ra(A)\subseteq \Ra(B) ~\Leftrightarrow P_BA=A \quad \text{and} \quad \Ra(A^*)\subseteq \Ra(B^*) ~\Leftrightarrow AQ_B=A, 
\end{equation}
partitioning $A$ in conformation with partition
     of $B,$ say $A=U_1\begin{bmatrix}
      A_1 & A_2 \\
       A_3 & A_4 \\
        \end{bmatrix}V_1^*,$
we get
\[\Ra(A)\subseteq \Ra(B) ~
\Leftrightarrow \begin{bmatrix}
       I_r & 0\\
       0 & 0\\
     \end{bmatrix}\begin{bmatrix}
       A_1 & A_2\\
       A_3 & A_4\\
     \end{bmatrix}=\begin{bmatrix}
       A_1 & A_2\\
       A_3 & A_4\\
     \end{bmatrix} ~\Leftrightarrow~  A_3=0,~ A_4=0.\]
Also, $\Ra(A^*)\subseteq \Ra(B^*)$ is equivalent to \[ \begin{bmatrix}
       A_1 & A_2\\
       0 & 0\\
     \end{bmatrix}\begin{bmatrix}
       I_r & 0\\
       0 & 0\\
     \end{bmatrix}=\begin{bmatrix}
       A_1 & A_2\\
       0 & 0\\
     \end{bmatrix} ~\Leftrightarrow~ A_2=0.\]
Thus, $A=U_1\begin{bmatrix}                                                                     A_1 & 0 \\
0 & 0 \\
\end{bmatrix}V^*_1.$  
Now, let $A_1=U_2\begin{bmatrix}
                 \Sigma K & \Sigma L  \\
                 0 & 0 \\
               \end{bmatrix}U^*_2$ be the Hartwig-Spindelbok decomposition of $A_1$, where $U_2\in \Cnn$ is unitary, $\Sigma$  is the diagonal matrix of non null singular values of $A_1$,  and $K \in \C^{t\times t}$, $L \in \mathbb{C}^{t\times (n-t)}$ satisfy $KK^* +LL^* = I_t$ with $t=\ra(A_1)$.\\
Define \begin{center}
$U=U_1 \begin{bmatrix}
              U_2 & 0 \\
              0 & I_{m-r} \\
              \end{bmatrix} $ and $V=V_1\begin{bmatrix}
             U_2& 0 \\
              0 & I_{n-r} \\
              \end{bmatrix} $. 
\end{center}Then $U$ and $V$ are unitary and
               \[U^*AV=\begin{bmatrix}
                       \Sigma K & \Sigma L & 0\\
                       0 & 0 & 0 \\
                       0 & 0 & 0
                     \end{bmatrix}\quad\text{and}\quad U^*BV= 
                     \begin{bmatrix}
                       D_1 & D_2 & 0\\
                       D_3 & D_4 & 0 \\
                       0 & 0 & 0
                     \end{bmatrix},\]
where $\begin{bmatrix}
                       D_1 & D_2\\
                       D_3 & D_4
                     \end{bmatrix}= U^*_2DU_2$. Note that $\rk(B)= \rk\left( \begin{bmatrix}
                       D_1 & D_2 \\
                       D_3 & D_4 
                     \end{bmatrix}\right)$ because $\rk(B)=\rk(D)=r$.  \\
Let $C_1=\Sigma K$ and $C_2=\Sigma L$. As $KK^* +LL^* = I_t$ we have $C_1 C_1^*+C_2 C_2^*=\Sigma^2$, and therefore $C_1 C_1^*+C_2 C_2^*$ is  nonsingular. 
Now, condition  $AA^*A=AB^*A$ is true if and only if following conditions simultaneously hold:\\
(i) $(C_1 C_1^*+C_2 C_2^*)C_1= (C_1 D_1^*+C_2 D_2^*)C_1$;\\
(ii) $(C_1 C_1^*+C_2 C_2^*)C_2= (C_1 D_1^*+C_2 D_2^*)C_2$.\\
Postmultiplying (i) by $C_1^*$, (ii) by $C_2^*$, and adding them we obtain \[(C_1 C_1^*+C_2 C_2^*)^2=(C_1 D_1^*+C_2 D_2^*)(C_1 C_1^*+C_2 C_2^*),\] whence  $C_1 C_1^*+C_2 C_2^*=C_1 D_1^*+C_2 D_2^*$. 
\\ (b)$\Rightarrow$(a) It is an immediate consequence of Definition \ref{def BaHa} and \eqref{charact ord space}. 
\end{proof}

We finish this section with a result involving the orthogonal projectors associated with the predecessors and successors in the partial orders considered.

\begin{theorem}  Let $A,B\in \Cm$. Then the following statements hold:
\begin{enumerate}[(a)]
%\item If $A\spa B$ if and only if $P_A\spa P_B$.
\item If $A\st B$ then $P_A\st P_B$.
\item If $A\minus B$ then $P_A\minus P_B$.
\item If $A\diam B$ then $P_A\diam P_B$.
\end{enumerate}
\end{theorem}
\begin{proof} 
%(a) It follows directly by definition of space pre-order and   the fact that $\Ra(M)=\Ra(P_M)$ for an arbitrary matrix $M$. \\
(a) It is clear from \eqref{star minus space} and the equivalence $M\subseteq L \Leftrightarrow P_MP_L=P_LP_M=P_M$, where $M$ and $L$ are two arbitrary subspaces.\\
The proofs of (b) and (c) are similar to (a).
\end{proof} 
The reverse implications of the above theorem are not true. For example, take the matrices
 
\begin{equation*}
A=\left[\begin{array}{c c}
1 &1\\
0 &0
\end{array}\right] \quad \text{and}\quad B=\left[\begin{array}{c c}
1 &0 \\
0 & 1
\end{array}\right].
\end{equation*}
A simple computation leads to
\begin{equation*}
P_A=\left[\begin{array}{c c}
1 &0\\
0 &0
\end{array}\right] \quad \text{and}\quad P_B=\left[\begin{array}{c c}
1 &0 \\
0 & 1
\end{array}\right].
\end{equation*}
As the orders $\st$, $\minus$, and $\diam$ are equivalent on the set of orthogonal projectors, it is sufficient to check that $P_A\st P_B$ holds, but $A\st B$ is not true.

\begin{remark}Note that $A\spa B$ if and only if $P_A\spa P_B$. In fact, it follows directly from the definition of space pre-order and the fact that $\Ra(M)=\Ra(P_M)$ for an arbitrary matrix $M$.
\end{remark}
\section{A connection with the left and right star partial orders}

To generalize the definition of the star partial order, Baksalary and Mitra \cite{BaMi} proposed the left star and right star orders, defined respectively as:
\begin{eqnarray}
A \lstar B &~\Leftrightarrow~& A^*A= A^* B ~~\text{and}~~ \Ra(A)\subseteq \Ra(B), \label{left star} \\
A \rstar B &~\Leftrightarrow~& AA^*=BA^* ~~\text{and}~~ \Ra(A^*)\subseteq \Ra(B^*). \label{right star} \nonumber
\end{eqnarray}

In \cite{LePaTh}, it was proved that the left star (resp., right star) partial order implies the diamond partial order under an additional condition:
\begin{eqnarray}
A\lstar B \text{ and } \Ra(A^*)\subseteq \Ra(B^*) &\Rightarrow&  A \diam B, \label{left star implies diamond} \\
A\rstar B ~\text{ and } ~\Ra(A)\subseteq \Ra(B) &\Rightarrow&  A \diam B. \label{right star implies diamond}
\end{eqnarray}

We will prove that the conditions on column spaces in \eqref{left star implies diamond} and \eqref{right star implies diamond} can be omitted. Moreover, we will prove that the left star and right star orders are located between the star and diamond orders in the sense that

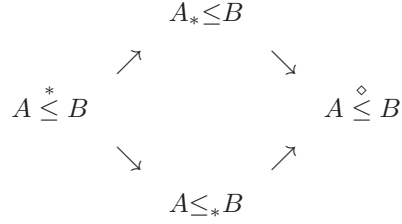
\[\begin{array}{c c c c c}
 & & A\lstar B & &\\
 & \nearrow & &\searrow & \\
 A\st B & & & & A\diam B \\
  & \searrow & & \nearrow &\\
 & & A\rstar B & &
\end{array}\]

\begin{theorem} \label{theorem implications} Let $A,B\in \Cm$.  Consider the following statements:
\begin{enumerate}[(a)]
\item $A\diam B$;
\item $A\lstar B$;
\item $A\rstar B$.
\end{enumerate}
Then (b)$\Rightarrow$(a) and (c)$\Rightarrow$(a).
\end{theorem}
\begin{proof}
(b)$\Rightarrow$(a) From \eqref{left star} we deduce $A^*A=B^*A$, whence $\Ra(A^*)=\Ra(A^*A)\subseteq \Ra(B^*)$ and $AA^*A=AB^*A$. In consequence, as $\Ra(A)\subseteq \Ra(B)$, we conclude that $A\spa B$ and $AA^*A=AB^*A$, that is, $A\diam B$. \\
(c)$\Rightarrow$(a) It is similar to the proof of (b)$\Rightarrow$(a).
\end{proof}

Notice that the converse implications in the above theorem are not true, which can be seen, for example, by using the matrices 
\[A=\begin{bmatrix}
   1 & 0 \\
   0 & 0
\end{bmatrix}
~~\text{and}~~
B=\begin{bmatrix}
   1 & 1 \\
   1 & -1 
\end{bmatrix}.\]  

In the following two results, we give some conditions under which the converse implications of Theorem \ref{theorem implications} are true.

\begin{theorem} Let $A,B\in \Cm$. Then the following statements are equivalent:
\begin{enumerate}[(a)]
\item $A\lstar B$;
\item $A\diam B$ and $A^*A=A^*B$;
\item $A\diam B$ and $A^\dag A=A^\dag B$;
\item $A\diam B$ and $A^*B$ is Hermitian.
\end{enumerate}
\end{theorem}
\begin{proof}
(a)$\Rightarrow$(b) Follows from Theorem \ref{theorem implications} and \eqref{left star}.  \\
(b)$\Rightarrow$(a) As $A\diam B$, from \eqref{star diam space} we obtain $A\spa B$, and so $\Ra(A)\subseteq \Ra(B)$. Now, the implication follows from \eqref{left star}. \\
(b)$\Leftrightarrow$(c) Notice that $A^*A=A^*B$ holds if and only if $\Ra(A-B)\subseteq \Nu(A^*)=\Nu(A^\dag)$, or equivalently, $A^\dag(A-B)=0$. Thus, $A^*A=A^*B$ is true if and only if $A^\dag A=A^\dag B$. \\
(b)$\Rightarrow$(d) Trivial. \\
(d)$\Rightarrow$(b) Since $A^*B$ is Hermitian, we have $AA^*B= AB^*A =AA^*A$, or equivalently $AA^*(B-A)=0$. So, $\Ra(B-A) \subseteq \Nu(AA^*)=\Nu(A^*)$, which implies $A^*(B-A)=0$. This completes the proof.
\end{proof}

\begin{theorem} Let $A,B\in \Cm$. Then the following statements are equivalent:
\begin{enumerate}[a)]
\item $A\rstar B$;
\item $A\diam B$ and $AA^*=BA^*$;
\item $A\diam B$ and $AA^\dag=BA^\dag$;
\item $A\diam B$ and $BA^*$ is Hermitian.
\end{enumerate}
\end{theorem}

\section{The diamond partial order on the set of square matrices}

By using the Hartwig-Spindelböck decomposition it is well known that every matrix $B\in \Cnn$ of rank $r>0$ can be represented in the form
\begin{equation}\label{HS of B}
B=U\left[\begin{array}{cc}
\Sigma K & \Sigma L\\
0 &0
\end{array}\right] U^*,
\end{equation}
where $U\in \Cnn$ is unitary, $\Sigma=\text{diag}(\sigma_1 I_{r_1},\sigma_2 I_{r_2},\dots, \sigma_t I_{r_t})$ is the diagonal matrix of singular values of $B$, $\sigma_1>\sigma_2>\cdots>\sigma_t>0$, $r_1+r_2+\cdots+r_t=r$, and $K \in \Crr$, $L \in \mathbb{C}^{r\times (n-r)}$ satisfy $KK^* +LL^* = I_r$. In this case, the Moore-Penrose inverse of $B$ is given by 
\begin{equation}\label{MP of B}
B^\dag=U\left[\begin{array}{cc}
K^*\Sigma^{-1} & 0\\
L^* \Sigma^{-1} & 0
\end{array}\right]U^*.
\end{equation}

The predecessors of a given matrix under the diamond partial order on the set of square matrices are found in \cite{MaRuTh} by using decomposition \eqref{HS of B}.

\begin{theorem}\cite [Theorem 6] {MaRuTh}\label{HS diamond order} Let $B\in \Cnn$ be a nonnull matrix written as in \eqref{HS of B}.Then the following statements are equivalent:
\begin{enumerate}[(a)]
\item there exists a matrix $A\in \Cnn$ such that $A\diam B$;
\item there exists an unique idempotent matrix $T\in \Crr$ such that 
\begin{equation}\label{HS of A}
A=U\left[\begin{array}{cc}
(\Sigma^{-1}T)^\dag K & (\Sigma^{-1}T)^\dag L\\
0&0
\end{array}\right]U^*.
\end{equation}
\end{enumerate}
\end{theorem}
The next result provides the Moore-Penrose of a matrix $A$ as in \eqref{HS of A}. Before, recall that Hung and Markham \cite{HuMa} proved that the Moore-Penrose inverse of a partitioned matrix $M=\left[\begin{array}{cc}
P & Q\\
0 & 0
\end{array}\right]$ is of the form 
\begin{equation}\label{MP of M}
M^\dag=\left[\begin{array}{cc}
P^* R^\dag & 0 \\
Q^* R^\dag & 0
\end{array}\right], ~\text{where}~ R=PP^*+QQ^*.
\end{equation}

\begin{theorem} \label{thm MP of A} Let $A\in \Cnn$ be a matrix as in \eqref{HS of A}. The Moore-Penrose inverse of $A$ is given by
\begin{equation}\label{MP of A}
A^\dag=U\left[\begin{array}{cc}
K^*\Sigma^{-1}T & 0\\
L^*\Sigma^{-1}T & 0
\end{array}\right]U^*.
\end{equation}
\end{theorem}
\begin{proof}
Consider $A$ written as in \eqref{HS of A}. Clearly,
\begin{equation*} \label{core EP eq3 }
 A^\dag= U\left[\begin{array}{cc}
(\Sigma^{-1}T)^\dag K & (\Sigma^{-1}T)^\dag L\\
0&0
\end{array}\right]^\dag U^*.
\end{equation*}
By applying \eqref{MP of M}, we have
\begin{equation}\label{PQR}
A^\dag = U\left[\begin{array}{cc}
P^* R^\dag & 0 \\
Q^* R^\dag & 0
\end{array}\right]U^*,
\end{equation}
where $R=PP^*+QQ^*$, $P=(\Sigma^{-1}T)^\dag K$, and $Q=(\Sigma^{-1}T)^\dag L$.
Now, we calculate $R$ as follows
\begin{equation}\label{R}
\begin{split}
  R &=  
 \left[\begin{array}{cc}
P & Q
\end{array}\right] \left[\begin{array}{cc}
P & Q
\end{array}\right]^*  \\
& = 
\left[
\begin{array}{cc}
(\Sigma^{-1}T)^\dag K & (\Sigma^{-1}T)^\dag L
\end{array}
\right]
\left[
\begin{array}{c}
K^* ((\Sigma^{-1}T)^\dag)^* \\
L^* ((\Sigma^{-1}T)^\dag)^* 
\end{array} 
\right]  \\
& = 
(\Sigma^{-1}T)^\dag \Sigma  \left[
\begin{array}{cc}
K & L
\end{array}
\right]
\left[
\begin{array}{c}
K^* \\
L^*  
\end{array} 
\right] ((\Sigma^{-1}T)^\dag)^* \\
& = 
(\Sigma^{-1}T)^\dag ((\Sigma^{-1}T)^\dag)^*. 
\end{split}
\end{equation}
Now, by using the identity $C^\dag = C^* (CC^*)^\dag$ with $C=(\Sigma^{-1}T)^\dag$, from \eqref{PQR} and \eqref{R}  we obtain
\[
A^\dag = U\left[\begin{array}{cc}
K^* ((\Sigma^{-1}T)^\dag)^* ((\Sigma^{-1}T)^\dag ((\Sigma^{-1}T)^\dag)^*)^\dag & 0 \\
L^* ((\Sigma^{-1}T)^\dag)^* ((\Sigma^{-1}T)^\dag ((\Sigma^{-1}T)^\dag)^*)^\dag  & 0
\end{array}\right]U^*=U\left[\begin{array}{cc}
K^*((\Sigma^{-1}T)^\dag)^\dag & 0\\
L^*((\Sigma^{-1}T)^\dag)^\dag  & 0
\end{array}\right]U^*,
\]
whence follows \eqref{MP of A}.
\end{proof}

Greville \cite{BeGr} came up with a necessary and sufficient condition for the reverse order law $(BA)^\dag=B^\dag A^\dag$ in the 1960s. Multiple studies have been done on this problem over the years. Next, we provide a characterization of the reverse order law when $A$ is below $B$ under the diamond partial order.

\begin{theorem} \label{ROL} Let $A,B\in \Cnn$ be matrices such that $A\diam B$. Suppose that $A$ and $B$ satisfy the conditions given in Theorem \ref{HS diamond order}. Then $(AB)^\dag=B^\dag A^\dag$ if and only if $((\Sigma^{-1}T)^\dag K\Sigma)^\dag=\Sigma^{-1}K^*\Sigma^{-1}T$.
\end{theorem}

\begin{proof}
Let $A$ and $B$ are of the form \eqref{HS of A} and \eqref{HS of B}, respectively, a straightforward computation yields
\begin{equation*}\label{HS of AB}
AB=U\left[\begin{array}{cc}
(\Sigma^{-1}T)^\dag K\Sigma K & (\Sigma^{-1}T)^\dag K\Sigma L\\
0 &0
\end{array}\right]U^*.
\end{equation*} 
As in the proof of Theorem \ref{thm MP of A}, it is easy to obtain 
\begin{equation}\label{MP of AB}
(AB)^\dag=U\left[\begin{array}{cc}
K^*((\Sigma^{-1}T)^\dag K\Sigma)^\dag & 0\\
L^* ((\Sigma^{-1}T)^\dag K\Sigma )^\dag& 0
\end{array}\right]U^*.
\end{equation}
On the other hand, from \eqref{HS of A} and \eqref{HS of B}, we have
\begin{equation} \label{B+A+}
B^\dag A^\dag =U\left[\begin{array}{cc}
(K^*\Sigma^{-1})^2 T &0\\
L^*\Sigma^{-1}K^*\Sigma^{-1}T &0
\end{array}\right]U^*.
\end{equation}
Now, from \eqref{MP of AB} and \eqref{B+A+} it follows that $(AB)^\dag=B^\dag A^\dag$ holds if and only if
\begin{equation}\label{Eq 1}
K^*((\Sigma^{-1}T)^\dag K\Sigma)^\dag = K^*\Sigma^{-1}K^*\Sigma^{-1}T,
\end{equation}
\begin{equation}\label{Eq 2}
L^* ((\Sigma^{-1}T)^\dag K\Sigma )^\dag=L^*\Sigma^{-1}K^*\Sigma^{-1}T.
\end{equation}
Finally, note that the above conditions are equivalent to  $((\Sigma^{-1}T)^\dag K\Sigma)^\dag =\Sigma^{-1}K^*\Sigma^{-1}T$. In fact,pre-multiplying \eqref{Eq 1} and \eqref{Eq 2} by $K$ and $L$, respectively, and then adding, we deduce $((\Sigma^{-1}T)^\dag K\Sigma)^\dag =\Sigma^{-1}K^*\Sigma^{-1}T$. Reciprocally, if $((\Sigma^{-1}T)^\dag K\Sigma)^\dag =\Sigma^{-1}K^*\Sigma^{-1}T$ it is clear that \eqref{Eq 1} and \eqref{Eq 2} are fulfilled. 
\end{proof}

An interesting consequence of the previous theorem arises when $A=B$. In this case, we obtain a new characterization of bi-dagger matrices using the Hartwig-Spindelb\"ock decomposition. Recall that $B$ is a bi-dagger matrix if $(B^2)^\dag=(B^\dag)^2$.

\begin{corollary} Let $B\in\Cnn$ be as in \eqref{HS of B}. Then $B$ is bi-dagger if and only if $(\Sigma K\Sigma)^\dag=\Sigma^{-1}K^*\Sigma^{-1}$.
\end{corollary}
\begin{proof}
It suffices to note that when $A=B$ in Theorem \ref{ROL}, the matrix $T$ reduces to the identity matrix.   
\end{proof}

Drazin \cite{Dr2} observed that the Moore-Penrose inverse is isotonic with respect to the star partial order, that is, $A\st B \Leftrightarrow A^\dag\st B^\dag$. However, a similar property for the diamond partial order is not generally true. In fact, by using \eqref{diamond characterization} one can see that Example \ref{example 1} and Example \ref{example 2} show that neither of the implications $A\diam B \Rightarrow A^\dag \diam B^\dag$ and $A^\dag \diam B^\dag \Rightarrow A\diam B$ is valid in general.

Next, we present some conditions under which $A\diam B$ implies $A^\dag \diam B^\dag$. 

\begin{theorem} Let $A,B\in \Cnn$ be matrices such that $A\diam B$. Suppose that $A$ and $B$ satisfy the conditions given in Theorem \ref{HS diamond order}. Then $A^\dag \diam B^\dag$ if and only if $T(T^*-I_r)\Sigma^{-2}T=0$.
\end{theorem} 
\begin{proof}
By definition of diamond partial order, $A^\dag \diam B^\dag$ if $A^\dag \spa B^\dag$ and $A^\dag (A^\dag)^*A^\dag=A^\dag (B^\dag)^* A^\dag$. As $A\diam B$, it is clear that $A^\dag \spa B^\dag$ is always satisfied. So, it remains to prove the second condition. 
In fact, from \eqref{MP of A} and the fact that $KK^*+LL^*=I_r$, we have 
\begin{equation}\label{Eq 3}
\begin{split}
A^\dag [(A^\dag)^*A^\dag] &= 
U\left[\begin{array}{cc}
K^*\Sigma^{-1}T & 0\\
L^*\Sigma^{-1}T & 0
\end{array}\right]\left[\begin{array}{cc}
T^*\Sigma^{-1}(KK^*+LL^*)\Sigma^{-1}T &0\\
0 & 0
\end{array}\right]U^*\\
&= U\left[\begin{array}{cc}
K^*\Sigma^{-1}T & 0\\
L^*\Sigma^{-1}T & 0
\end{array}\right]\left[\begin{array}{cc}
T^*\Sigma^{-2}T &0\\
0 & 0
\end{array}\right]U^*\\
&= U\left[\begin{array}{cc}
K^*\Sigma^{-1}TT^*\Sigma^{-2}T &0\\
L^*\Sigma^{-1}TT^*\Sigma^{-2}T  & 0
\end{array}\right]U^*.
\end{split}
\end{equation}
Similarly, from \eqref{MP of B}, direct calculus yields
\begin{equation}\label{Eq 4}
A^\dag(B^\dag)^*A^\dag=U\left[\begin{array}{cc}
K^*\Sigma^{-1}T\Sigma^{-2}T  & 0\\
L^*\Sigma^{-1}T\Sigma^{-2}T  & 0
\end{array}\right]U^*.
\end{equation}
In consequence, from \eqref{Eq 3} and \eqref{Eq 4} we deduce that $A^\dag (A^\dag)^*A^\dag=A^\dag (B^\dag)^* A^\dag$ is equivalent to
\begin{equation}\label{Eq 5}
K^*\Sigma^{-1}TT^*\Sigma^{-2}T=K^*\Sigma^{-1}T\Sigma^{-2}T,
\end{equation}
\begin{equation}\label{Eq 6}
L^*\Sigma^{-1}TT^*\Sigma^{-2}T=L^*\Sigma^{-1}T\Sigma^{-2}T.
\end{equation}
Finally, note that the above two conditions are equivalent to $T(T^*-I_r)\Sigma^{-2}T=0$. In fact, by pre-multiplying \eqref{Eq 5} and \eqref{Eq 6} by $K$ and $L$, respectively, and then adding, we obtain $\Sigma^{-1}TT^*\Sigma^{-2}T=\Sigma^{-1}T\Sigma^{-2}T$, which in turn is equivalent to $T(T^*-I_r)\Sigma^{-2}T=0$ because $\Sigma$ is nonsingular. Clearly, if this last equality is true, then \eqref{Eq 5} and \eqref{Eq 6} are fulfilled.
\end{proof}

\section*{Funding}

This paper was partially supported by the Universidad Nacional de R\'io Cuarto, Argentina (grant PPI 18/C559), the Universidad Nacional de La Pampa, Facultad de Ingenier\'ia  (Grant Resol. Nro. 135/19), and CONICET (grant PIBAA 28720210100658CO). 

\section*{Acknowledgements}

One of the authors, Gayatri Maharana, would like to express her sincere gratitude to Prof. Jajati Keshari Sahoo for his unwavering support, guidance, and mentorship throughout the course of this research.

\end{document}